\documentclass[12pt,reqno]{amsart}

\usepackage{amssymb}
\usepackage{amsfonts}
\usepackage{amsthm}

\theoremstyle{plain} 
\newtheorem{theorem}{Theorem}

\theoremstyle{definition} 

\theoremstyle{remark}

\newcommand{\R}{\ensuremath{\mathbb{R}}}

\newcommand{\T}{\ensuremath{\mathbb{T}}}
\newcommand{\N}{\ensuremath{\mathbb{N}}}

\numberwithin{equation}{section}
\numberwithin{theorem}{section}

\small\normalsize

\begin{document}
\author[anderson]{Douglas R. Anderson}
\title[second--order linear equations]{Second--order linear constant coefficient dynamic equations with polynomial forcing on time scales}
\address{Department of Mathematics and Computer Science, Concordia College, Moorhead, MN 56562 USA\\visiting United International College, Zhuhai, Guangdong, China}
\email{andersod@cord.edu}
\urladdr{http://www.cord.edu/faculty/andersod/bib.html}


\begin{abstract}
A general solution for a second-order linear constant coefficient dynamic equation with polynomial forcing on time scales is given.
\end{abstract}

\maketitle

\thispagestyle{empty}


\section{polynomial forcing of second-order ordinary dynamic equations with constant coefficients}
Recall that polynomials on time scales \cite{ab} are defined via $h_0(t,a)\equiv 1$ and recursively for $k\in\N$ by 
$$ h_k(t,a)=\int_{a}^{t}h_{k-1}(\tau,a)\Delta\tau. $$
We have the following result.


\begin{theorem}[Constant Coefficients with Polynomial Forcing]
Let $a\in\T$ and $\alpha,\beta\in\R$. A general solution of the second-order linear constant coefficient equation with polynomial forcing 
\begin{equation}\label{polyf} 
 y^{\Delta\Delta}+2\alpha y^\Delta+\beta y = \sum_{i=0}^{k}\gamma_ih_i(\cdot,a), \quad \gamma_i\in\R,
\end{equation}
with $\beta\ne 0$, $\beta\ne \alpha^2$, and $1-2\alpha\mu+\beta\mu^2\ne 0$ is given by
\begin{equation}\label{polysol}
 y(t) = c_1e_{\lambda_1}(t,a)+c_2e_{\lambda_2}(t,a)+\sum_{i=0}^k \xi_ih_i(t,a), 
\end{equation}
where $\lambda_1=-\alpha-\sqrt{\alpha^2-\beta}$, $\lambda_2=-\alpha+\sqrt{\alpha^2-\beta}$, $\xi_{k-1}=\frac{1}{\beta}\left(\gamma_{k-1}-\frac{2\alpha}{\beta}\gamma_k\right)$, $\xi_{k}=\frac{1}{\beta}\gamma_k$,
$$ \xi_i = \omega_1 \lambda_1^i + \omega_2 \lambda_2^i + \sum_{\tau=0}^{i-1}\sum_{s=0}^{\tau-1} \gamma_s\lambda_1^{i+s-2\tau}\beta^{\tau-1-s} \quad\text{for}\quad i\in\{0,\cdots,k-2\}, $$
and 
\begin{eqnarray*} 
 \begin{bmatrix} \omega_1 \\ \omega_2 \end{bmatrix} &=& \frac{1}{2\beta^{k-1}\sqrt{\alpha^2-\beta}} \left[\begin{array}{rr} \lambda_2^k & -\lambda_2^{k-1} \\ -\lambda_1^k & \lambda_1^{k-1} \end{array}\right] \begin{bmatrix} \dfrac{1}{\beta}\left(\gamma_{k-1}-\dfrac{2\alpha}{\beta}\gamma_k\right) -\displaystyle\sum_{\tau=0}^{k-2}\sum_{s=0}^{\tau-1}\gamma_s\lambda_1^{k-1+s-2\tau}\beta^{\tau-1-s} \\ \dfrac{1}{\beta}\gamma_k -\displaystyle\sum_{\tau=0}^{k-1}\sum_{s=0}^{\tau-1}\gamma_s\lambda_1^{k+s-2\tau}\beta^{\tau-1-s} \end{bmatrix}.
\end{eqnarray*}
\end{theorem}

\begin{proof}
By the theory of linear dynamic equations \cite[Theorem 3.16]{bp1}, a general solution of \eqref{polyf} has the form
\begin{equation}
 y(t) = c_1e_{\lambda_1}(t,a)+c_2e_{\lambda_2}(t,a)+y_d(t) 
\end{equation}
for 
\begin{equation}\label{lambs}
 \lambda_1=-\alpha-\sqrt{\alpha^2-\beta} \quad\text{and}\quad \lambda_2=-\alpha+\sqrt{\alpha^2-\beta}, 
\end{equation} 
where the time scale exponential functions $e_{\lambda_1}(\cdot,a)$ and $e_{\lambda_2}(\cdot,a)$ are solutions of the corresponding homogeneous equation
$$  y^{\Delta\Delta}+2\alpha y^\Delta+\beta y = 0, $$
and $y_d$ is a particular solution of the inhomogeneous equation \eqref{polyf}. Here $y_1=e_{\lambda_1}(\cdot,a)$ and $y_2=e_{\lambda_2}(\cdot,a)$ are linearly independent since their Wronskian
$$ W(y_1,y_2)=y_1y_2^\Delta-y_1^\Delta y_2=(\lambda_2-\lambda_1)y_1y_2=2\sqrt{\alpha^2-\beta}\;e_{\lambda_1\oplus\lambda_2}(\cdot,a)=2\sqrt{\alpha^2-\beta}\;e_{(-2\alpha+\mu\beta)}(\cdot,a) $$
is well defined due to the regressivity assumption $1-2\alpha\mu+\beta\mu^2\ne 0$ and is never zero on $\T$.
We guess that a particular solution $y_d$ has the form
$$ y_d = \displaystyle\sum_{i=0}^{k}\xi_ih_i(\cdot,a), \qquad \xi_i\in\R. $$
Then 
$$ y_d^\Delta = \sum_{i=0}^{k-1}\xi_{i+1}h_i(\cdot,a) \quad\text{and}\quad y_d^{\Delta\Delta} = \sum_{i=0}^{k-2}\xi_{i+2}h_i(\cdot,a). $$
Plugging into \eqref{polyf} and solving for the coefficients we have
\begin{equation}\label{diffeterm}
 \xi_{k-1}=\frac{1}{\beta}\left(\gamma_{k-1}-\frac{2\alpha}{\beta}\gamma_k\right) \quad\text{and}\quad \xi_{k}=\frac{1}{\beta}\gamma_k, 
\end{equation}
while the rest of the coefficients $\xi_i$ satisfy the second-order linear constant coefficient difference equation with forcing given by
\begin{equation}\label{diffegam}
 \xi_{i+2}+2\alpha\xi_{i+1}+\beta\xi_i=\gamma_i, \quad i=0,\ldots,k-2. 
\end{equation}
The preceding difference equation \eqref{diffegam} has the solution \cite[Section 4]{at}
$$ \xi_i = \omega_1 \lambda_1^i + \omega_2 \lambda_2^i + \sum_{\tau=0}^{i-1}\sum_{s=0}^{\tau-1} \gamma_s\lambda_1^{i+s-2\tau}\beta^{\tau-1-s} $$
where $\lambda_1$ and $\lambda_2$ are given in \eqref{lambs} above. Using the terminal conditions in \eqref{diffeterm} and simple linear algebra we have 
\begin{eqnarray*} 
 \begin{bmatrix} \omega_1 \\ \omega_2 \end{bmatrix} &=& \frac{1}{2\beta^{k-1}\sqrt{\alpha^2-\beta}} \left[\begin{array}{rr} \lambda_2^k & -\lambda_2^{k-1} \\ -\lambda_1^k & \lambda_1^{k-1} \end{array}\right] \begin{bmatrix} \dfrac{1}{\beta}\left(\gamma_{k-1}-\dfrac{2\alpha}{\beta}\gamma_k\right) -\displaystyle\sum_{\tau=0}^{k-2}\sum_{s=0}^{\tau-1}\gamma_s\lambda_1^{k-1+s-2\tau}\beta^{\tau-1-s} \\ \dfrac{1}{\beta}\gamma_k -\displaystyle\sum_{\tau=0}^{k-1}\sum_{s=0}^{\tau-1}\gamma_s\lambda_1^{k+s-2\tau}\beta^{\tau-1-s} \end{bmatrix}.
\end{eqnarray*}

\end{proof}


\end{document}